\documentclass[12pt]{amsart}

\usepackage{graphicx}
\usepackage{amsmath,amsthm,euscript}
\usepackage{amsfonts}
\usepackage{amssymb}
\usepackage{a4wide}

\newtheorem{thm}{Theorem}[section]
\newtheorem{lem}[thm]{Lemma}

\newtheorem{prop}[thm]{Proposition}

\theoremstyle{definition}

\newtheorem{quest}{Question}

\newtheorem{notn}[thm]{Notation}

\begin{document}
\title{Primitive algebraic algebras of polynomially bounded growth}
\author{Jason P. Bell}
\thanks{Research supported by NSERC grant 31-611456.}
\keywords{primitive algebras, Kurosh problem, algebraic algebras, Gelfand-Kirillov dimension}
\subjclass[2000]{16P90}

\address{Jason Bell\\
Department of Mathematics\\
Simon Fraser University\\
Burnaby, BC, V5A 1S6\\ CANADA
}

\email{jpb@math.sfu.ca}

\author{Lance W. Small}
\address{Lance Small\\
Department of Mathematics\\
University of California, San Diego\\
La Jolla, CA, 92093-0112\\
 USA}
\email{lwsmall@math.ucsd.edu}

\author{Agata Smoktunowicz}
\thanks{The research of the third
author was supported by Grant No. EPSRC
 EP/D071674/1.}
\address{Agata Smoktunowicz\\
Maxwell Institute for Mathematical Sciences\\
School of Mathematics, University of Edinburgh\\
James Clerk Maxwell Building, King's Buildings, Mayfield Road\\
Edinburgh EH9 3JZ, Scotland\\
 UK}
 \email{A.Smoktunowicz@ed.ac.uk}

\bibliographystyle{plain}

\date{Oct. 1, 2010}
\begin{abstract}
We show that if $k$ is a countable field, then there exists a finitely generated, infinite-dimensional, primitive algebraic $k$-algebra $A$ whose Gelfand-Kirillov dimension is at most six.  In addition to this we construct a two-generated primitive algebraic $k$-algebra.  We also pose many open problems.
\end{abstract}
\maketitle
\section{Introduction}
In recent years there has been renewed interest in the
construction of algebraic algebras that are not locally
finite-dimensional; that is, algebras that are algebraic over
their base fields, but which have the property that some finitely
generated subalgebra is infinite-dimensional.  The first
construction of such an algebra was made by Golod and Shafarevich
\cite{GS}, which provided a counterexample to a famous conjecture
of Kurosh \cite{Kur}.  The third-named author \cite{Sm1, Sm2, Sm3}
has produced a variety of important counterexamples to conjectures
about algebraic algebras, which go beyond the original
construction of Golod and Shafarevich.

There is an intimate connection between Kurosh's conjecture in ring theory and its group-theoretic counterpart, Burnside's conjecture.  Indeed, Golod and Shafarevich were able to use their construction to produce a finitely generated infinite torsion group, providing the first counterexample to the Burnside problem.  Gromov \cite{Gr} showed that finitely generated groups of polynomially bounded growth are nilpotent-by-finite.  In particular, there cannot exist a finitely generated infinite torsion group of polynomially bounded growth.  In light of this theorem, it is natural to ask whether an analogous result holds for rings.  Surprisingly, Lenagan and Smoktunowicz \cite{LS} showed that over a countable field there exists a finitely generated infinite-dimensional algebraic algebra whose Gelfand-Kirillov dimension (defined in \S2) is at most $20$.  By refining estimates used in this paper,
Lenagan, Smoktunowicz, and Young \cite{LSY} showed that the bound of $20$ on the Gelfand-Kirillov dimension could be lowered to three.

Despite all of the recent progress on Kurosh-type problems, there has been little progress---either with or without restrictions on Gelfand-Kirillov dimension---on the related problem of whether or not there exists an algebraic division ring that is finitely generated and infinite-dimensional as an algebra over its center.  In fact, most of the constructions have had a large, nil Jacobson radical and are very far from being division rings.  A resolution of this problem appears to be far away, but one can nevertheless hope to make some progress by producing counter-examples in small increments that are in some sense increasingly closer to being division rings.  As a partial step towards producing a finitely generated algebraic infinite-dimensional division ring, one can ask whether or not there is a finitely generated infinite-dimensional primitive algebraic algebra.  This question was initially posed by Kaplansky \cite[Problem 15]{K} and answered by the first- and second-named authors \cite{BS}.  In this paper, we show that one can in fact construct such an example in which the Gelfand-Kirillov dimension is at most six.
\begin{thm} \label{thm: main} Let $k$ be a countable field.
 Then there exists a finitely generated infinite-dimensional algebraic primitive $k$-algebra $A$ whose Gelfand-Kirillov dimension is at most six.
\end{thm}
In order to do this construction, we use an (unpublished) affinization construction due to the second-named author, along with growth estimates of the first-named author \cite{B}, and the construction of Lenagan, \emph{et al.} \cite{LSY}.

In addition to this, we also consider the question of whether or not there are primitive algebra algebras that are two-generated.  We note that the algebra constructed in Theorem \ref{thm: main}, while finitely generated, needs several generators.  It is natural to consider whether one can construct such algebras with fewer generators.  We are able to prove the following result.
\begin{thm} \label{thm: main2}  Let $k$ be a countable field.  Then there is an infinite-dimensional primitive algebraic $k$-algebra that is generated by two
 elements.
 \end{thm}

The outline of the paper is as follows.  In \S 2, we define Gelfand-Kirillov dimension and describe some basic facts about it.
In \S 3, we give a very general affinization construction, which allows one to construct finitely generated algebras with many different properties.
 In \S 4, we apply our construction and use the construction of Lenagan \emph{et al.} to prove Theorem \ref{thm: main}.
 In \S 5 we describe another affinization construction and prove Theorem \ref{thm: main2}. Finally, in \S 6 we pose some open problems which are related to our investigations.
\section{Gelfand-Kirillov dimension}
In this section, we define Gelfand-Kirillov dimension and give some basic facts about it.
Given a field $k$ and a finitely generated $k$-algebra $A$, the \emph{Gelfand-Kirillov dimension} of $A$ (GK dimension, for short) is defined to be
$${\rm GKdim}(A) \ := \ \limsup_{n\rightarrow\infty}~ \log\big({\rm dim}~V^n\big)\big/\log n,$$
where $V$ is a finite dimensional $k$-subspace of $A$ which
contains the identity of $A$ and which generates $A$ as an
$k$-algebra.  As it turns out, Gelfand-Kirillov dimension is
independent of choice of $V$ \cite[pp. 6]{KL}.  In the case that
$A$ is not finitely generated, the GK dimension of $A$ is defined
to be the supremum of the GK dimensions of finitely generated
subalgebras of $A$.

We give some of the basic properties about Gelfand-Kirillov dimension.
\begin{prop} Gelfand-Kirillov dimension has the following properties:
\begin{enumerate}
\item the Gelfand-Kirillov dimension of a finitely generated
    commutative algebra is the same as its Krull dimension
    \cite[pp. 39]{KL};
\item there are no algebras whose GK dimension is strictly
    between $0$ and $1$;
\item there are no algebras whose GK dimension is strictly
    between $1$ and $2$ \cite[pp. 18]{KL};
\item for every $\alpha\ge 2$ there exists a finitely
    generated algebra of GK dimension $\alpha$ \cite[pp.
    162]{KL};
\item if $A$ and $B$ are algebras such that $B$ is either a
    finite left or right $A$-module, then ${\rm GKdim}(A)={\rm
    GKdim}(B)$ \cite{BK};
\item if $A$ and $B$ are two algebras then ${\rm
    GKdim}(A\otimes B)\le {\rm GKdim}(A)+{\rm GKdim}(B)$
    \cite[pp. 28]{KL}.
\end{enumerate}
\end{prop}
We refer the reader to the book of Krause and Lenagan \cite{KL} for additional facts about Gelfand-Kirillov dimension.

\section{Affinization}
In this section, we describe a general affinization construction.  This construction takes a countably generated algebra $T$ over a field $k$ and builds a finitely generated $k$-algebra $A$ with the property that $eAe\cong T$ for some idempotent $e$ of $A$.  Our construction has the additional property that we can bound the Gelfand-Kirillov dimension of the algebra $A$ in terms of the Gelfand-Kirillov dimension of $T$.

\begin{notn} \label{notn: 1}
Throughout this section, we fix the following notation:
\begin{enumerate}
\item we let $k$ be a field;
\item we let $T$ be a prime, countably generated $k$-algebra;
\item we let $B$ be a prime infinite-dimensional finitely generated $k$-algebra;
\item we let $R=B\star k[y]$, the free product of $B$ and $k[y]$ in the category of finitely generated $k$-algebras;
\item we let
\[ S \ =\  \left( \begin{array}{cc} k + Ry & R \\ Ry & R \end{array}
\right). \]
\end{enumerate}
\end{notn}
The ring $S$ is generated as a $k$-algebra by
\[ \left( \begin{array}{cc} 0 & 0 \\ 0 & c \end{array}\right),
\left( \begin{array}{cc} 1 & 0 \\ 0 & 0 \end{array}\right),
\left( \begin{array}{cc} 0 & 1 \\ 0 & 0 \end{array} \right),
~{\rm and}~ \left( \begin{array}{cc} 0 & 0 \\ y & 0 \end{array}
\right), \] where $c$ ranges over elements in a finite generating set for the $k$-algebra $B$.
Hence $S$  is a finitely generated $k$-algebra.

Given these data, we show that we can construct a finitely generated prime $k$-algebra $A$ with the following properties:
\begin{enumerate}
\item $A$ has an idempotent $e$ such that
$$eAe\cong T;$$
\item $A/AeA$ is a homomorphic image of $B$;
\item ${\rm GKdim}(A)\le 2{\rm GKdim}(B)+{\rm GKdim}(T).$
\end{enumerate}
We call such a ring an \emph{affinization} of $T$ with respect to $B$.

To describe this construction, note that $k+Ry$ is
a free $k$-algebra on the infinitely many generators $\{ cy~ |~ c\in \mathcal{C}\}$, where $\mathcal{C}$ is a basis for $B$ as a $k$-vector space.
It follows that we have a surjective ring homomorphism
\begin{equation}
\Phi:
k+Ry\  \rightarrow\
T. \label{eq: Phi}
\end{equation}
  Let
\begin{equation}
P\ =\  {\rm ker}(\Phi)
\label{eq: P}
\end{equation}
and let
$e_{i,j}$ denote the matrix with a $1$ in the $(i,j)$-entry and zeros
everywhere else.  Notice $P$ is
a prime ideal. Observe that
$Q' := S(Pe_{1,1})S$ satisfies $e_{1,1}Q'e_{1,1}=P$.   Using Zorn's lemma we can
choose an ideal $Q$ in $S$ maximal with respect to the property that
\begin{equation} e_{1,1}
Q e_{1,1} \ =\
\left( \begin{array}{cc} P & 0 \\ 0 & 0 \end{array} \right).
\label{eq: Q} \end{equation}
By maximality, we have that $Q$ is prime.  Throughout this paper, we let
$\overline{e_{i,j}}$ denote the image of $e_{i,j}$ is $S/Q$.

Observe that $Q$ is in fact uniquely determined.
To see this, suppose that $Q'$ is another such ideal.  Then
$$e_{1,1}(Q+Q')e_{1,1} \ = \ e_{1,1}Qe_{1,1} + e_{1,1}Q'e_{1,1} \ = \ P$$  By maximality of $Q$
and $Q'$ we have that $Q=Q+Q'=Q'$ and so $Q=Q'$.  We note that
\begin{equation}
 Q \ \supseteq\
 \left( \begin{array}{cc} 0 & 0 \\ Ry & 0 \end{array}\right)
\left( \begin{array}{cc} P & 0 \\ 0 & 0 \end{array}\right)
\left( \begin{array}{cc} 0 & 0 \\ 0 & R \end{array}\right)\  = \
\left( \begin{array}{cc} 0 & 0 \\ 0 & RyPR \end{array}\right).
 \label{eq: e22}\end{equation}
Similarly, \begin{equation} Q \ \supseteq\
 PRe_{1,2},\label{eq: e12}\end{equation}
and
\begin{equation} Q\ \supseteq\  \label{eq: e21} RyPe_{2,1}. \end{equation}
The algebra $S/Q$ has the property that
\begin{equation}
\overline{e_{1,1}}(S/Q)\overline{e_{1,1}}\ \cong\  T.
\end{equation}
We call $S/Q$ the \emph{affinization of} $T$ \emph{with respect to} $\Phi$ and $B$,
 and we denote it
by $\mathcal{A}(T,B,\Phi)$.  Since the prime ideal $Q$ is uniquely determined by $\Phi$,
the algebra $\mathcal{A}(T,B,\Phi)$ is uniquely determined by $T$, $B$, and $\Phi$.

Observe that if $e$ denotes the image of $e_{1,1}$ in $S/Q=\mathcal{A}(T,B,\Phi)$,
then by construction we have:
\begin{enumerate}
\item $\mathcal{A}(T,B,\Phi)$ is prime;
\item $e\mathcal{A}(T,B,\Phi)e\cong T$;
\item $\mathcal{A}(T,B,\Phi)/AeA$ is a homomorphic image of
    $B$.
\end{enumerate}
We now
 prove our main result of this section.

\begin{prop}
\label{thm: main1} Assume the notation given in Notation \ref{notn: 1}.
There exists a homomorphism
$\Phi : k +Ry \rightarrow T$ such that $\mathcal{A}(T,B,\Phi)$ has
Gelfand-Kirillov dimension at most $2{\rm GKdim}(B)+{\rm GKdim}(T)$.
\end{prop}

\begin{proof} We let $\alpha$ and $\beta$ denote respectively the Gelfand-Kirillov dimension of $T$ and $B$.  If $\alpha$ or $\beta$ is infinite, there is nothing to prove, thus it is no loss of generality to assume that $\alpha,\beta<\infty$.

Let $P$ and $Q$ be as in equations (\ref{eq: P}) and
(\ref{eq: Q}).  Let
$W_0$ be a finite-dimensional subspace of $B$ that contains $1_B$ and generates $B$ as a $k$-algebra, let
 \begin{equation}
W \ =\  W_0+ky\ \subseteq \ R,
\end{equation}
and let
$V\subseteq S$ be the generating subspace of $S$ given by
\[ V \ \subseteq\  \left( \begin{array}{cc} k
+Wy & W \\ Wy & W
 \end{array}
\right). \]

We have
\[ V^n \ \subseteq\  \left( \begin{array}{cc} k
+W^{n}y & W^{n} \\ W^{n}y & W^n
 \end{array}
\right). \]
We shall construct a homomorphism that will give an affinization with the desired upper bound on the Gelfand-Kirillov dimension.  Let $\EuScript{B}=\{1, u_1,u_2\ldots \}\subseteq T$
be a basis for
$T$ as a $k$-vector
space.

For each $j\ge 1$, define $U_j$ to be the vector space spanned by the
first $j+1$ elements of $\EuScript{B}$; that is,
\begin{equation} U_j\ = \ k+ku_1+\cdots +ku_j.\end{equation}
Since $T$ has GK dimension $\alpha$, we have
$$\limsup_{n\rightarrow\infty}
~ \log \bigl({\rm dim}~ (U_j)^n\bigr)\Big/\log n\  \le \  \alpha.$$
Hence there exists a positive integer
$m_j$ such that $${\rm dim}~ (U_j)^n\  <\  n^{\alpha+1/j}$$
 for all $n\ge m_j$.  By increasing $m_j$ if necessary, we may assume
that $m_j\ge m_{j-1}$ for all $j\ge 2$.
Pick $v_j\in W_0^{m_j}\setminus W_0^{m_j-1}$ for each $j\ge 1$.  By construction, $v_1,v_2,\ldots $ are linearly independent and thus can be extended to a basis $\mathcal{C}$ for $B$.
We define $$\Phi: k+Ry \   \rightarrow\  T$$ by
\[ \Phi\big(cy\big) \ =\
 \left\{ \begin{array}{ll} u_j & {\rm if~}c=v_{j}~{\rm for~some
~}j\ge 1, \\
0 & {\rm if~}c\in \mathcal{C}\setminus \{v_1,v_2,\ldots \}, \end{array} \right. \]
and we extend by linearity.

Consider
\[ {\rm dim} \left( \begin{array}{cc} k+ W^n y & 0 \\ 0 & 0 \end{array}
\right). \]
Let $\varepsilon > 0$.  Then there exists some number $j_0$ such that
$${\rm dim}(W^n)<n^{\beta+\varepsilon}$$ for all $n\ge m_{j_0}$.

Suppose $m_{j} \le n < m_{j+1}$ for some $j\ge j_0$.  Then since $Pe_{1,1}= e_{1,1}Qe_{1,1}$,
an element of $W^ny\overline{e_{1,1}}$ is determined by its behavior modulo $Pe_{1,1}$.
Since $n<m_{j+1}$, we have $$\Phi\big(W^ny\big)\ \subseteq \  (U_j)^n.$$
Hence $${\rm dim}~ W^ny\overline{e_{1,1}}\  \le\  {\rm dim}~ (U_j)^n
\ \le\  n^{\alpha + 1/j}.$$

  We now compute the dimension of $W^n\overline{e_{2,2}}$.  Notice that any element of
$R$
can be expressed as a linear combination of elements of $\mathcal{C}$, elements of the form $c_1yc_2$, with $c_1,c_2\in \mathcal{C}$, and elements
of the form
$c_1ywyc_2$, where $w$ is a word over the alphabet $\mathcal{C}\cup \{y\}$, and $c_1,c_2\in \mathcal{C}$.

Hence
anything in $W^n$ is contained in
$$W_0^n+ W_0^nyW_0^n+
W_0^nyW^nyW_0^n.$$  Thus
$${\rm dim}(W^n)\le {\rm dim}(W_0^n) + \left({\rm dim}(W_0^n)\right)^2 +
 \left({\rm dim}(W_0^n)\right)^2 \cdot {\rm dim}(W^ny).$$

Observe that $RyPRe_{2,2}\subseteq Q$ and hence the image in $\overline{e_{2,2}}\mathcal{A}(T,B,\Phi)
\overline{e_{2,2}}$ of an element of the form
$c_1ywyc_2e_{2,2}$, with $c_1,c_2\in \mathcal{C}$ and $w$ a word over the alphabet $\mathcal{C}\cup \{y\}$, is completely
determined by the behavior of $wy$ mod $P$.
As $\Phi\big(W^ny\big)\subseteq (U_j)^n$, we have for $n\ge m_j$,
$${\rm dim}~W_0^nyW^nyW_0^n \overline{e_{2,2}}\ \le\
 {\rm dim}(U_j)^n{\rm dim}(W_0^n)^2 \  \le\  n^{\alpha+2\beta+2\varepsilon+1/j}.$$  Since
Thus
$${\rm dim}~ W^n\overline{e_{2,2}} \ \le\  n^{\beta+\varepsilon} + n^{2\beta+2\varepsilon}+n^{\alpha+2\beta+\varepsilon+1/j} \ =\  {\rm O}(n^{2\beta+\alpha+2\varepsilon+1/j}).$$

As $j\to\infty$ as $n\to \infty$ and $\varepsilon>0$ is arbitrary, we see that
$${\rm dim}~ W^n\overline{e_{2,2}} = {\rm O}(n^{2\beta+\alpha+\varepsilon})$$ for every $\varepsilon>0$

Let $D$ denote the ``diagonal'' of $\mathcal{A}(T,B,\Phi)$ and let
$C$ denote the ``upper-triangular part'' of
$\mathcal{A}(T,B,\Phi)$.  We have just shown
that
$$D\ \cong \ \big((k+Ry)/P\big)\oplus \big(R/RyPR\big)$$
 has GK dimension at most
$\alpha +2\beta+\varepsilon$.

  Observe that $C=D+\overline{e_{1,2}}D$ and hence $B$ has GK dimension
at most
$\alpha+2\beta+\varepsilon$ \cite[Lemma 4.3]{KL}.  Finally, note that
$$\mathcal{A}(T,B,\Phi)\ =\ B+B(y\overline{e_{2,1}})$$ and thus
$$\mathcal{A}(T,B,\Phi)$$ has GK dimension at most $\alpha+2\beta+\varepsilon$,
\cite[Lemma 4.3]{KL}.
 Since $\varepsilon>0$ is arbitrary, we conclude that $\mathcal{A}(T,B,\Phi)$
has GK dimension at most $\alpha+2\beta$.

 \end{proof}
 \section{Algebraic algebras}
 In this section we prove Theorem \ref{thm: main}
 \begin{proof}[Proof of Theorem \ref{thm: main}.]
 Note that Lenagan, Smoktunowicz, and Young \cite{LSY} have shown that one can construct a finitely generated
 infinite-dimensional algebraic algebra $A$ of Gelfand-Kirillov dimension at most three over any countable field
 $k$. We note that this algebra has a prime infinite-dimensional homomorphic image $B$.
 Indeed, the prime radical (the intersection of all prime ideals) of an arbitrary algebra is always locally
 nilpotent and hence has Gelfand-Kirillov dimension zero.  Consequently the prime radical doesn't equal the whole algebra $A$.
  Note that finitely dimensional nil algebras are nilpotent,
  therefore $A/I$ is infinite dimensional for any prime ideal in
  $A$.
   The algebra $B$ is necessarily algebraic, finitely generated, and has  Gelfand-Kirillov dimension at most three.

We let $T$ be a countably generated infinite-dimensional primitive
$k$-algebra of Gelfand-Kirillov dimension zero. We note that an
example of such an algebra is given by the first- and second-named
authors \cite{BS}.

By Proposition \ref{thm: main1}, there exists a prime finitely generated $k$-algebra $A$ with the following properties:
\begin{enumerate}
\item ${\rm GKdim}(A)\le 2{\rm GKdim}(B)+{\rm GKdim}(T)\le 6$;
\item there is an idempotent $e\in A$ such that $eAe\cong T$;
\item $A/AeA$ is a homomorphic image of $B$.
\end{enumerate}
The second property gives that $A$ as primitive \cite[Theorem 1]{LRS}, as $A$ is a prime ring with a primitive corner.

We next claim that $AeA$ is a locally finite two-sided ideal of $A$.  To see this, note that any finite-dimensional subspace of $AeA$ is contained in a subspace of the form $W e W$ for some finite-dimensional subspace of $A$.

Then $$\left( W eW\right)^m \subseteq W(e W^{2}e)^{m-2} e W.$$
As $eAe\cong T$, we see that $e W^{2}e\cong W'$ for some finite-dimensional subspace $W'$ of $T$.  As $T$ is locally finite, we have that $(W')^p = (W')^{p+1}$ for some natural number $m$ and hence $$\left( W^n eW^n\right)^{p+2}=\left( W^n eW^n\right)^{p+3},$$ giving that $AeA$ is locally finite.

Since $A/AeA$ is a homomorphic image of $B$, it is algebraic and
$AeA$ is a locally finite two-sided ideal, we see that $A$ is
algebraic. The result follows. \end{proof}

\section{Affinization with two generators}

 In this section, we briefly describe another affinization construction. This construction
is a generalization of a construction of Markov \cite{Mar}.  We rely heavily on both the ideas and notation from the recent construction of Lenagan,
Smoktunowicz, and Young \cite{LSY}.  Using these ideas we are able to construct an infinite-dimensional primitive algebraic algebra generated by just two elements.

 We point out that our construction only works over countable fields.
    \begin{notn} \label{notn: 2}
Throughout this section, we fix the following notation:
\begin{enumerate}
\item we let $k$ be a countable field;
\item we let $T$ be a prime, countably generated $k$-algebra
    with unity;
\item  let $k\{x,y\}$ denote the free $k$-algebra on two generators;
\item we let $B$ denote an infinite-dimensional $k$-algebra of the form $B=k\{x,y\}/(y^{2},I)$ where
  $I\subseteq (x,xyx)k\{x,xyx\}\subseteq k\{x,y\}$.
\end{enumerate}
\end{notn}

\begin{thm} \label{thm: A}
 Assume the notation from Notation \ref{notn: 2}. Then there exists
  $k$-algebra $A$ generated by two elements $x$ and $y$ such that
  that $y^{3}=y^{2}$, $y^{2}Ay^{2}\cong T$ and $A/(y^{2})$ is a homomorphic image of $B$.
\end{thm}
\begin{proof}  Let $R=k\{x,y\}/(y^{3}-y^{2},I)$. Note that
 $(y^{3}-y^{2},I)$ denotes the ideal of $k\{x,y\}$ generated by
$I$ and by the element $y^{3}-y^{2}$.

 Let $r_{1}, r_{2}, \ldots $ be a basis for
$\sum_{0\leq i,j \leq n}k\{x,xyx\}/I\cap k\{x,xyx\}$ as a $k$-vector space.

Observe that the elements $y^{2}r_{1}y^{2}, y^{2}r_{2}y^{2}\ldots $
 are generators of a countably-generated unital free noncommutative $k$-algebra, which we denote by $C$.

It follows that we have a surjective ring
homomorphism
\begin{equation}
\Phi:
C\  \rightarrow\
T. \label{eq: Phi1}
\end{equation}
  Let
\begin{equation}
P\ =\  {\rm ker}(\Phi)
\label{eq: P1}
\end{equation}
  Notice $P$ is a prime
ideal in $C$. Let $Q'=RPR+RP+PR+P$.  Then $Q'$ is an ideal in $R$
and $y^{2}Q'y^{2}=P$. Using Zorn's lemma we can choose an ideal
 $Q$ in $R$ maximal with the property
that $y^{2}Qy^{2}=P$. Then $Q$ is a prime ideal, because if
$Q\subseteq Q_{1}$ and $Q\subseteq Q_{2}$ for some ideals $Q_{1},
Q_{2}$ in $R$ then $Q_{1}Q_{2}\subseteq Q$ gives
$y^{2}Q_{1}y^{2}Cy^{2}Q_{2}y^{2}\subseteq y^{2}Qy^{2}=P$, and
so either $y^{2}Q_{1}y^{2}=P$ or $y^{2}Q_{2}y^{2}=P$, as $P$
is a prime ideal of $C$.

 Observe also that $C=y^{2}Ry^{2}$, because $y^{2}-1$ annihilates the homogeneous maximal ideal of $C$ when we regard $C$ as a subalgebra of $R$.
  Let $A=R/Q$. Observe that by construction we have
 $A=R/Q$ is prime; furthermore, $A$ is generated by elements $x,y$ and $A/(y^{2})$ is a homomorphic image of
 $B$.  Finally, we clearly have $y^{2}Ay^{2}\cong T$.  The result follows.
 \end{proof}
We next prove a technical lemma.  All of the groundwork needed for this result was done by Lenagan, Smoktunowicz, and Young \cite{LSY}.  In order to avoid unnecessary repetition, we will make use of the notation and proofs from their paper and a careful reading of this paper is essential for a full understanding of this lemma.
\begin{lem}\label{lem: LSY}
Let $k$ be a countable field.  Then there exists an infinite-dimensional $k$-algebra $R$
 generated by two elements $x,y$ such that $y^{2}=0$ and with the property that the ideal $(x,y)$ in $R$ is nil.  Furthermore, one can choose $R$ to have the property that its GK dimension is at most $3$.
 \end{lem}
\begin{proof} We use the construction and notation from
the paper of Lenagan, Smoktunowicz, and Young \cite{LSY}. Let
$V_{i}, U_{i}, H(n)$ be as in Theorem $3$ from this paper. Observe
that if we take \begin{equation} V_{1}=\{x,y\},\qquad
U_{1}=\{\emptyset \}\end{equation} and
 \begin{equation}V_{2}=\{xx,xy\}, \qquad U_{2}=\{yy,yx \}\end{equation}
 and
 \begin{equation} V_{4}=\{xxxx,xxxy\}, \end{equation}
 \begin{equation}
U_{4}=\{y^4, yyxy, yyxx, yyyx, xyyy, xxyy, yxyy, yxxx, yxxy, yxyx, xyyx, xxyx,xyxy,xyxx \}\end{equation} and apply Theorem 3 of \cite{LSY}, then $$H(i)y^{2}H(8-i-2)\subseteq
 U(4)H(4)+H(4)U(4)$$ and thus $y^{2}\in E$ where $E$ is the ideal
 defined in \cite{LSY} with the property that the image of the ideal $(x,y)$ in $k\{x,y\}/E$ is nil
and $k\{x,y\}/E$ is an algebra of GK dimension not exceeding three. Thus the image of $y^2$ in $k\{x,y\}/E$ is zero.
\end{proof}
\begin{prop} \label{prop: B} Let $k$ be a countable field.   Then there is a two-sided ideal
$I$ of $k\{x,y\}$ satisfying the following properties:
\begin{enumerate}
\item $I$ is generated by elements from $(x,xyx)k\{x,xyx\}$;
\item the $k$-algebra $B= k\{x,y\}/(y^{2},I)$ is infinite-dimensional as a $k$-vector space;
\item the image of the ideal $(x,y)$ in $B$ is nil.
\end{enumerate}\end{prop}
\begin{proof}
 Let $R$ be the algebra defined in Lemma \ref{lem: LSY}.
 Then $R=k\{x,y\}/E$ for some ideal $E$ of $k\{x,y\}$. Because the image of the homogeneous maximal ideal $(x,y)$ of $k\{x,y\}$ is nil in $R$, there is $p$ such that $x^{p}\in
 E$.  Since $N$ is infinite-dimensional as a $k$-vector space and $y^{2}\in E$ we see that
 $p>2$.

  Let $$B=k\{x,y\}/(\{xEx ,y^{2}, x^{p}, xyx^{i}Ex^{j}yx\, :\, 0\geq i,j<p\}).$$  Then $B$ satisfies the conclusion of the statement of the proposition.  \end{proof}
\begin{proof}[Proof of Theorem \ref{thm: main2}] Let $T$ be infinitely generated locally finite
primitive algebra with unity over $k$. Let $B$ be an algebra
satisfying the conclusion of Proposition \ref{prop: B}.  By Theorem \ref{thm: A},
there exists prime algebra  $A$ generated by two
 elements $x,y$ such that $y^{4}=y^{2}$ and $y^{2}Ay^{2}$ is
 isomorphic to $T$.  Since $y^2$ is idempotent and $T$ is primitive, we see that $A$ is primitive by a result of Lanski, Resco, Small \cite{LRS}.
It only remains to show that $A$ is
 algebraic. Note that $A/(y^{2})$ is a homomorphic image of $B$ and hence is algebraic.  Furthermore $y^{2}$ generates a
 locally finite ideal of $A$, and hence $A$ is algebraic over $k$.  The result follows.
\end{proof}

 \section{Questions}
 In this section, we pose some questions related to algebraic algebras and our main result.
  \begin{quest} Does there exist some real number $\alpha>2$ such that for every $\beta\ge \alpha$ there exists a finitely generated nil algebra whose GK dimension is exactly $\beta$?
 \end{quest}

 It is the opinion of the authors that by suitably modifying the homomorphism $\phi$ which is used in our construction, one should in fact be able to construct finitely generated nil primitive algebras of every GK dimension larger than or equal to six.

 \begin{quest} Does there exist a finitely generated, infinite-dimensional $k$-algebra that is a division ring of finite GK dimension?
 \end{quest}

Even answering this question for algebras of quadratic growth would be an impressive result.  In fact, there are no known examples of division rings that are algebraic over their centers that do not have the property that each finitely generated subalgebra is finite-dimensional over its center.

 \begin{quest} Does there exist a finitely generated infinite-dimensional unital simple algebraic algebra of finite GK dimension?
 \end{quest}

 We have constructed a finitely generated infinite-dimensional unital primitive algebra of finite GK dimension.  The next logical step is to attempt to modify this construction somehow to create a simple algebra with the aforementioned properties.  Unfortunately, there is an obstacle that one immediately encounters; namely, algebras constructed via the affinization method always have a nonzero proper two-sided ideal generated by the image of $ye_{2,2}$; moreover, any homomorphic image under which this ideal becomes zero is a homomorphic image of $B$.
The third-named author \cite{Sm3}, on the other hand, has constructed a simple nil algebra
 (which clearly cannot be unital).  Observe that by Nakyama's lemma a finitely generated Jacobson radical algebra cannot be simple.

 \begin{quest} Does there exist a finitely presented infinite-dimensional algebraic algebra of finite GK dimension?
 \end{quest}

 We note that all constructions of infinite-dimensional algebra algebras so far have required infinite sets of relations in order to obtain algebraic algebras.  It is conjectured that the corresponding question for groups (with the polynomial growth restriction removed), namely, whether or not there exists a finitely presented infinite torsion group, has a negative answer.  A negative answer to the above question would both lend verisimilitude to the group theoretic conjecture and may also provide techniques which could eventually be used to prove this conjecture.

 \begin{quest} Does there exist a finitely generated infinite-dimensional algebraic algebra of finite GK dimension over an uncountable base field? \end{quest}

 One of the main problems with trying to do the Lenagan and Smoktunowicz construction over an uncountable field $k$ is that a countable enumeration of the elements of the free $k$-algebra on two generators is required.  This can be relaxed somewhat, but it does not seem possible to modify their construction to obtain an  infinite-dimensional algebraic algebra of finite GK dimension over an uncountable base field.
 \begin{quest} Does there exist a finitely generated algebraic algebra of GK dimension two?
 \end{quest}
 Lenagan, Smoktunowicz, and Young \cite{LSY} produced a finitely generated infinite-dimensional algebraic algebra (over a countable base field) whose GK dimension is at most three.  It should be noted that finitely generated algebras of GK dimension strictly less than $2$ satisfy a polynomial identity by Bergman's gap theorem \cite[Theorem 2.5]{KL} along with a theorem of Small, Stafford, and Warfield \cite{SSW}.  Consequently, the question of what is the infimum over all Gelfand-Kirillov dimensions of finitely generated infinite-dimensional algebraic algebras is still unresolved.
\begin{quest}
Can one give a construction as done in Theorem \ref{thm: main2} when $k$ is uncountable?  Can one modify the construction in Theorem \ref{thm: main2} to give an algebra of finite GK dimension?
\end{quest}
In the authors' opinion, the first part should not be too
difficult, but it will probably require a deep understanding of
the paper \cite{LSY}.  The second part will require delicate
estimates, but is probably quite doable, however the calculations
involved appear to be more complicated and less elegant than those
used in the construction of Theorem \ref{thm: main1}.
 \begin{quest} Are there other interesting algebras of low GK dimension that can be constructed using Proposition \ref{thm: main1}?
 \end{quest}
As an example, let $k$ be a field and let $T=k[x_1,x_2,\ldots]/I$, where $I$ is the ideal generated by $\{x_i^{i+1}~:~i\ge 1\}$.  Then $T$ has prime radical $J$ generated by the image of $(x_1,x_2,\ldots)$, which is nil and not nilpotent.  Observe that $J$ is the sum of all nilpotent ideals of $T$ and hence $T$ has no maximal nilpotent ideal.  

Since $T$ has GK dimension $0$, we can apply Proposition \ref{thm: main1}, taking $B$ to be a nil ring of GK dimension at most $3$, as we did in the proof of Theorem \ref{thm: main}, to obtain an algebra $A$ whose GK dimension is at most six.   Moreover, by using basic facts about corners, one can see that $A$ is a finitely generated $k$-algebra with a non-nilpotent prime radical and without a maximal nilpotent ideal.  This example answers a question of Lvov \cite[Question 2.69]{FKS}.  In a similar manner, one should be able to construct strange examples of $2$-generated Jacobson radical algebras that are not nil.


\begin{thebibliography}{99}
 \bibitem{B} Jason P. Bell, Examples in finite Gelfand-Kirillov
 dimension, \emph{J. Algebra} 263 (2003), no. 1, 159--175.
 \bibitem{BS} Jason P. Bell, Lance W. Small, A question of
 Kaplansky, Special issue in celebration of Claudio Procesi's 60th
 birthday, \emph{J. Algebra} 258 (2002), no. 1, 386--388.
 \bibitem{BK} W.Bohro, H. Kraft, {\: U}ber die Gelfand-Kirillov
 Dimension, \emph{Math.Ann.} 220 (1) (1976), 1--24.
 \bibitem{FS} Daniel R. Farkas, Lance W. Small, Algebras which are
 nearly finite dimensional and their identities, \emph{Israel J.
 Math} 127 (2002), 245-251.
 \bibitem{FKS} V. T. Filipov, V.K. Kharchenko, I.P.Shestakov
 (Editors), Dniester Notebook, Unsolved Problems in the Theory of
 Rings and Modules, Mathematics Institute, Russian Academy of
 Sciences, Siberian Branch, Novosibirsk, Fourth Edition, 1993.
 \bibitem{GS} E.S. Golod and I.R. Shafarevich, On the class field tower,
\emph{Izv. Akad. Nauk. SSSR Mat. Ser.} {\bf 28} (1964), 261--272.
(in Russian)
\bibitem{Gr} M. Gromov, Groups of polynomial growth and expanding
maps, \emph{Publ. Math. IHES} 53 (1981), 53--73.
\bibitem{J} Nathan Jacobson, \emph{Structure of Rings}, Amer. Math. Soc. Coll.,
vol. 37, rev. ed., 1964.
\bibitem{K} Irving Kaplansky, ``Problems in the theory of rings'' revisited,
\emph{Amer. Math. Monthly} {\bf 77} (1970), no. 5, 445--454.
\bibitem{KL}  G\"unter R. Krause and Thomas H. Lenagan, \emph{Growth
of algebras and Gelfand-Kirillov dimension.} Revised edition. Graduate
Studies in Mathematics, 22.  American Mathematical Society, Providence,
RI, 2000.
\bibitem{Kur} A. Kurosh, Ringtheoretische Probleme die mit dem Burnsideschen
Problem \"uber periodische Gruppen in Zussammenhang stehen,
\emph{Bull. Acad. Sci. URSS. S\'er. Math. [Izvestia Akad. Nauk
SSSR]} {\bf 5} (1941), 223--240.
\bibitem{LRS} Charles Lanski, Richard Resco, and Lance Small, On the primitivity
of prime rings. \emph{J. Algebra} {\bf 59} (1979), no. 2,
395--398.
 \bibitem{LS} T. H. Lenagan, A. Smoktunowicz, An infinite
 dimensional affine nil algebra with finite Gelfand-Kirillov
 dimension, \emph{J.Amer. Math. Soc.} 20 (2007), no. 4, 989-1001.
 \bibitem{LSY} T. H. Lenagan, A. Smoktunowicz, and A. Young, Nil
 algebras with restricted growth, submitted.
 \bibitem{Mar} V. T. Markov, Some examples of finitely generated
 algebras, Uspiekhi Mat. Nauk {\bf 221} (1981), 185--186.
 \bibitem{SSW}  L. W. Small, J. T. Stafford, R. Warfield Jr.,
 Affine algebras of Gelfand-Kirillov dimension one are PI, \emph{Math. Proc. Cambridge Philos. Soc.} {\bf 97} (1985), no. 3, 407--414.
 \bibitem{Sm1} Agata Smoktunowicz, Graded algebras associated to
 algebraic algebras need not be algebraic. European Congress of
 Mathematics, \emph{Eur. Math. Soc.} Z{\: u}rich, 2010, 441--449.
 \bibitem{Sm2} Agata Smoktunowicz, Makar-Limanov's conjecture on
 free subalgebras, \emph{Adv. Math.} 222 (2009), no. 6,
 2107--2116.
 \bibitem{Sm3} Agata Smoktunowicz, A simple nil ring exists,
 \emph{Comm. Algebra} 30 (2002), no. 1, 27--59.

 \end{thebibliography}
\end{document}